\documentclass[11pt,twoside]{amsart}
\usepackage{amsxtra}
\usepackage{color}
\usepackage{amsopn}
\usepackage{amsmath,amsthm,amssymb}
\usepackage{enumerate}

\title{Coupled $\SU(3)$-structures and Supersymmetry}
\author{Anna Fino and Alberto Raf{}fero}
\thanks{The authors are supported by the Project PRIN ``Variet\`a reali e complesse: geometria, topologia e analisi armonica'', 
by the Project FIRB ``Geometria Differenziale e Teoria Geometrica delle Funzioni'', and by GNSAGA of INdAM}
\address{Dipartimento di Matematica ``G. Peano", Universit\`a di Torino, via Carlo Alberto 10, 10123 Torino, Italy}
\email{annamaria.fino@unito.it}
\email{alberto.raffero@unito.it}

\theoremstyle{remark}
\newtheorem{remark}{Remark}[section]

\theoremstyle{definition}
\newtheorem{Definition}[remark]{Definition}
\newtheorem{Example}[remark]{Example}

\theoremstyle{plain}
\newtheorem{Theorem}[remark]{Theorem}
\newtheorem{Proposition}[remark]{Proposition}
\newtheorem{Corollary}[remark]{Corollary}

\newcommand{\beq}{\begin{equation}}
\newcommand{\eeq}{\end{equation}}
\newcommand{\bqn}{\begin{eqnarray}}
\newcommand{\eqn}{\end{eqnarray}}
\newcommand{\bqne}{\begin{eqnarray*}}
\newcommand{\eqne}{\end{eqnarray*}}

\newcommand{\R}{{\mathbb R}}

\newcommand{\C}{{\mathbb C}}

\newcommand{\tz}{\tau_0}
\newcommand{\tu}{\tau_1}
\newcommand{\td}{\tau_2}
\newcommand{\ttr}{\tau_3}
\newcommand{\W}{\wedge}

\newcommand{\f}{\varphi}
\newcommand{\e}{\varepsilon}
\newcommand{\de}{\delta}

\newcommand{\psip}{\psi_+}
\newcommand{\psim}{\psi_-}

\newcommand{\SU}{{\rm SU}}
\newcommand{\ddt}{\frac{\partial}{\partial t}}
\newcommand{\Wd}{w_2^-}
\newcommand{\dt}{\frac{d}{dt}}
\newcommand{\lt}{\lambda_t}
\newcommand{\tut}{u_{t}}
\newcommand{\ttrt}{\eta_{t}}

\begin{document}
\maketitle
\begin{abstract} We review coupled $\SU(3)$-structures, also known in the literature as restricted half-flat structures, in relation to supersymmetry. In particular, we study special classes 
of examples admitting such structures and the behaviour of flows of $\SU(3)$-structures with respect to the coupled condition.
\end{abstract}
\medskip

\section{Introduction}
In the physical literature, manifolds endowed with $\SU(3)$-structures have been frequently considered to construct string vacua 
\cite{DGLM, GLL, Gur, GLM, KLM, Lar, LM, LT, Str, Tom}. 

In this paper, we are mainly interested in the class of $\SU(3)$-structures that are relevant for $\mathcal{N} = 1$ compactifications of type IIA string theory on spaces of the form 
${\rm AdS}_4\times N$, where ${\rm AdS}_4$ is the four-dimensional anti de Sitter space and $N$ is a six-dimensional compact smooth manifold. 
The requirement of $\mathcal{N}=1$ supersymmetry implies the existence of a globally defined complex spinor on the internal 6-manifold $N$. 
As a consequence, the structure group of $N$ reduces to $\SU(3)$, which is equivalent to the existence on $N$ of an almost Hermitian structure $(h,J,\omega)$  
and a complex (3,0)-form $\Psi$ of nonzero constant length satisfying some compatibility conditions. 
As shown in \cite{LT}, in the case where the two $\SU(3)$-structures are proportional, imposing the Killing 
spinor equations for four-dimensional $\mathcal{N}=1$ string vacua of type IIA on ${\rm AdS}_4$ constrains the intrinsic torsion of the $\SU(3)$-structure to lie in 
$\mathcal{W}_1^-\oplus\mathcal{W}_2^-$. Further constraints on the torsion forms are implied by the Bianchi identities for the background fluxes in the absence of sources. 
Moreover, all these constraints are not only necessary but also sufficient to guarantee the existence of solutions. 
Examples of this kind of solutions were considered for instance in \cite{CKKLTZ, KLT, LT, Tom}. 

$\SU(3)$-structures whose torsion class is $\mathcal{W}_1^-\oplus\mathcal{W}_2^-$ are known as {\it coupled} $\SU(3)$-structures \cite{Sal} in the mathematical literature 
and are characterized by the fact that they are {\it half-flat} $\SU(3)$-structures, i.e., both $\psip:=\Re(\Psi)$ and $\omega\W\omega$ are closed forms, 
having $d\omega$ proportional to $\psip$. Coupled structures were recently considered in \cite{FR, MS, Raf}. They are of interest for instance because their underlying almost 
Hermitian structure is {\it quasi-K\"ahler} and because they generalize the class of {\it nearly K\"ahler} $\SU(3)$-structures, namely the half-flat structures having $d\omega$ proportional 
to $\psip$ and $d\psim$ proportional to $\omega\W\omega$, where $\psim:=\Im(\Psi)$.  

Up to now, very few examples of manifolds admitting complete nearly K\"ahler structures are known. In the homogeneous case 
there are only finitely many of them by \cite{But}, while new complete inhomogeneous examples were recently found on $S^6$ and $S^3\times S^3$ in \cite{FH}.
Among the remarkable properties of nearly K\"ahler structures in dimension 6, it is worth recalling here that the Riemannian metric $h$ they induce is Einstein, 
that is, its Ricci tensor ${\rm Ric}(h)$ is a scalar multiple of $h$. 
It is then quite natural to ask whether coupled structures inducing Einstein metrics can exist or if requiring that a coupled structure induces an Einstein metric implies that it is actually 
nearly K\"ahler. An attempt to find coupled Einstein structures on explicit examples was done in \cite{Raf}, where the existence of invariant coupled Einstein structures was excluded on 
the compact manifold $S^3\times S^3$ for ${\rm Ad}(S^1)$-invariant Einstein metrics and on all the six-dimensional solvmanifolds. While writing this paper, we found out that the work 
\cite{Tom}, which provides a family of ${\rm AdS}_4$ vacua in IIA string theory, contains an example of a coupled Einstein structure. This answers to the question and can be used to 
construct examples of $G_2$-structures with non-vanishing torsion inducing Einstein and Ricci-flat metrics.

One of the main motivations to study half-flat structures is due to the role they play in the construction of seven-dimensional manifolds with holonomy contained in $G_2$. 
More in detail, by a result of Hitchin \cite{Hit}, on a 6-manifold $N$ it is possible to define a flow for $\SU(3)$-structures, the so-called {\it Hitchin flow}, which can be solved for any 
given analytic half-flat structure as initial condition. A solution to the flow equations consists of a family of half-flat structures depending on a parameter $t\in I\subseteq\R$ 
and allows to define a torsionless $G_2$-structure on the product manifold $I\times N$. 
One question that naturally arises is then whether coupled structures, which are in particular half-flat, are preserved by this flow. 

A generalization of the Hitchin flow can be introduced considering an $\SU(3)$-structure, not necessarily half-flat, and using it to define a $G_2$-structure with torsion on 
the product manifold $I \times N$. The evolution equations for the differential forms defining the $\SU(3)$-structure can then be obtained by requiring that the intrinsic torsion of the 
$G_2$-structure belongs to a certain torsion class. Of course, the Hitchin flow equations can be recovered as a special case of this generalized flow. 
This idea was considered for example in \cite{DLS}, where the generalized Hitchin flow was used as a tool to study the moduli space of $\SU(3)$-structure manifolds 
constituting the internal compact space for four-dimensional $\mathcal{N}=\frac12$ domain wall solutions of heterotic string theory. 
In that case, the authors considered the non-compact seven-dimensional manifold defined by combining the direction perpendicular to the domain wall and 
the internal 6-manifold and observed that it is possible to define on it a $G_2$-structure whose non-vanishing intrinsic torsion forms can be recovered using the results of \cite{GLL, LM}.

Furthermore, homogeneous spaces admitting coupled structures were used to provide examples of heterotic ${\mathcal N} = \frac12$ 
domain wall solutions with vanishing fluxes in \cite{KLM} and an attempt to generalize this result in a more general case was done in \cite{GLL}.

The present paper is organized as follows. 
In Section \ref{preliminaries} we review some definitions and properties regarding $\SU(3)$- and $G_2$-structures. In Section \ref{cpdsusy} we 
study coupled structures in relation to supersymmetry. In Section \ref{exsec} we describe some explicit examples and in Section \ref{flows} we study the behaviour of flows of 
$\SU(3)$-structures with respect to the coupled condition.
\ \\
\ \\
\noindent {\em{Acknowledgements}}.  The authors would like to thank Thomas Madsen for useful conversations.

\section{Review of $\SU(3)$-structures and $G_2$-structures}\label{preliminaries}\label{pre}
An $\SU(3)$-structure on a six-dimensional smooth manifold $N$ is the data of a Riemannian metric $h$, an orthogonal almost complex structure $J$, a 2-form $\omega$ 
related to $h$ and $J$ via the identity $\omega(\cdot,\cdot) = h(J\cdot,\cdot)$ and a $(3,0)$-form of nonzero constant length $\Psi=\psip+i\psim$ which is {\it compatible} with $\omega$, 
i.e.,
$$
\omega\W\Psi=0,
$$
and satisfies the {\it normalization condition}
$$
\frac{i}{2}\left(\Psi\W\overline\Psi\right) = \frac23\omega^3 = 4dV_h,
$$
where $dV_h$ is the Riemannian volume form of $h$. At each point $p\in N$ there exists an $h$-orthonormal frame $(e^1,\ldots,e^6)$ of $T^*_pN$, 
called {\it adapted frame} for the $\SU(3)$-structure, such that
\begin{equation}
\begin{array}{rcl}
\omega &=& e^{12}+e^{34}+e^{56},\\
\Psi &=& (e^1+ie^2)\W(e^3+ie^4)\W(e^5+ie^6),
\end{array}\nonumber
\end{equation}
and whose dual frame $(e_1,\ldots,e_6)$ is adapted for $J$, i.e.,
$$Je_i = e_{i+1}, \ i=1,3,5.$$
\begin{remark}
Here and hereafter, the notation $e^{ijk\cdots}$ is a shortening for the wedge product $e^i\W e^j \W e^k\W\cdots$. Moreover, we will also use the notation $\theta^n$ as a shortening 
for the wedge product of a differential form $\theta$ by itself for $n$-times.
\end{remark}

Using the results of \cite{Hit1, Rei}, one can show that an $\SU(3)$-structure actually depends only on the pair $(\omega,\psip)$, let us recall briefly how. 
For each $p\in N$, let $V:=T_pN$, denote by $A:\Lambda^5(V^*) \rightarrow V\otimes\Lambda^6(V^*)$ the canonical isomorphism given by $A(\xi) = v \otimes \Omega$, 
where $i_v\Omega = \xi$, and define for a fixed $\rho \in\Lambda^3(V^*)$
$$
K_\rho : V \rightarrow V\otimes\Lambda^6(V^*),\ \  K_\rho(v) = A((i_v \rho)\W\rho)
$$ 
and 
$$
\lambda : \Lambda^3(V^*) \rightarrow (\Lambda^6(V^*))^{\otimes2},\ \  \lambda(\rho) = \frac16{\rm tr}K^2_\rho.
$$
If $\lambda(\rho)\neq0$,  $\sqrt{|\lambda(\rho)|} \in \Lambda^6(V^*)$ defines a volume form by choosing the orientation of $V$ for which $\omega^3$ is positively oriented. Moreover, 
whenever $\lambda(\rho)<0$ the following endomorphism defines an almost complex structure 
$$
J_{\rho} := -\frac{1}{\sqrt{-\lambda(\rho)}}K_\rho.
$$
An $\SU(3)$-structure on $N$ can then be defined as a pair $(\omega,\psip)$ such that the 2-form $\omega$ is non degenerate, i.e., $\omega^3\neq0$, the 3-form $\psip$ 
is compatible with $\omega$ and satisfies $\lambda(\psip(p))<0$ for each $p\in N$, the almost complex structure is $J=J_{\psip}$, 
the imaginary part of $\Psi$ is given by $\psim:=J\psip$, the normalization condition holds and $h(\cdot,\cdot) :=\omega(\cdot,J\cdot)$ defines a Riemannian metric.

The {\it intrinsic torsion} $\tau$ of an $\SU(3)$-structure is completely determined by the exterior derivatives of $\omega, \psip, \psim,$ as shown in \cite{CS}. 
More in detail, we have
\begin{equation}\renewcommand\arraystretch{1.4}
\begin{array}{lcl}\label{hf}
d\omega &=& -\frac32w_1^-\psip +\frac32 w_1^+\psim + w_3 + w_4\W\omega,\\
d\psip &=&w_1^+\omega^2 -w_2^+\W\omega + w_5\W\psip,\\
d\psim &=& w_1^-\omega^2 - w_2^-\W\omega +Jw_5\W\psip,
\end{array}
\end{equation}
where $w_1^\pm\in C^\infty(N)$, $w_2^\pm\in\Lambda^{1,1}_0(N)$, $w_3\in\Lambda^{2,1}_0(N)$, $w_4,w_5\in\Lambda^1(N)$ are the {\it intrinsic torsion forms} 
of the $\SU(3)$-structure. 
It is then possible to divide the $\SU(3)$-structures in classes by seeing which torsion forms vanish. 
For example, if $\omega,\psip$ and $\psim$ are all closed, then all the torsion forms 
vanish and the manifold $N$ is {\it Calabi-Yau}. If all the torsion forms but $w_1^-$ vanish, the $\SU(3)$-structure is said to be {\it nearly K\"ahler} and we write 
$\tau\in\mathcal{W}_1^-$. If both $\psip$ and $\omega^2$ are closed, then the torsion forms $w_1^+, w_2^+, w_4,w_5$ vanish, the $\SU(3)$-structure is said to be {\it half-flat} and we 
write $\tau\in\mathcal{W}_1^-\oplus\mathcal{W}_2^-\oplus\mathcal{W}_3$. 
As recently shown in \cite{ACFH}, the $\SU(3)$-structures can also be described in terms of a characterizing spinor and the spinorial field equations it satisfies.

In \cite{BV}, it was shown that the Ricci and the scalar curvature of the metric $h$ induced by an $\SU(3)$-structure can be expressed in terms of the intrinsic torsion forms. 
In particular, if we consider the projections $E_1:\Lambda^2(N)\rightarrow\Lambda^{1,1}_0(N)$ and $E_2:\Lambda^3(N)\rightarrow\Lambda^{2,1}_0(N)$ given by 
$$\renewcommand\arraystretch{1.4}
\begin{array}{rcl}
E_1(\beta) &=&\frac12(\beta+J\beta)-\frac{1}{18}*((*(\beta+J\beta)+(\beta+J\beta)\W\omega)\W\omega)\omega,\\
E_2(\rho) &=& \rho-\frac12*(J\rho\W\omega)\W\omega-\frac14*(\rho\W\psim)\psip-\frac14*(\psip\W\rho)\psim,
\end{array}
$$
where $*$ is the Hodge operator defined using $h$ and the volume form $dV_h$,
then the traceless part of the Ricci tensor has the following expression
$${\rm Ric}^0(h) = \iota^{-1}(E_1(\phi_1)) + \gamma^{-1}(E_2(\phi_2)),$$
where the 2-form $\phi_1$ and the 3-form $\phi_2$ depend on the intrinsic torsion forms and their derivatives and the maps $\iota:S^2_+(N)\rightarrow\Lambda^{1,1}_0(N)$ and $
\gamma:S^2_-(N)\rightarrow\Lambda^{2,1}_0(N)$ are (pointwise) $\mathfrak{su}(3)$-modules isomorphisms (see \cite{BV} for the details).  
The Ricci tensor of $h$ can then be recovered from the identity 
$${\rm Ric}(h) = \frac16{\rm Scal}(h)h +  {\rm Ric}^0(h).$$

Starting from an $\SU(3)$-structure $(\omega,\psip)$ on a 6-manifold $N$,  it is possible to construct a $G_2$-structure on the 7-manifold 
$I\times N$, where $I\subseteq\R$ is a connected open interval. Before describing how, we recall that a $G_2$-structure on a  seven-dimensional manifold $M$ is characterized by 
the existence of a globally defined 3-form $\varphi$ 
inducing a Riemannian metric $g_{\varphi}$ and a volume form $dV_{g_\f}$ given by
\begin{equation}\label{gfi}
g_{\varphi} (X, Y) dV_{g_\f}=  \frac16i_X \varphi \wedge i_Y \varphi \wedge \varphi,
\end{equation}
for  any pair of vector fields $X,Y\in\mathfrak{X}(M)$.
The {\it intrinsic torsion} of a $G_2$-structure $\f$ is completely determined by the exterior derivatives of $\f$ and $*_\f\f$, where $*_\f$ is the Hodge operator defined using the metric 
$g_\f$ and the volume form $dV_{g_\f}$. More in detail, it holds \cite{Br}
\begin{equation}
\begin{array}{rcl}
d\f &=& \tz*_\f\f+3\tu\W\f+*_\f\ttr,\\
d*_\f\f&=& 4\tu\W*_\f\f+\td\W\f,
\end{array}
\end{equation}
where $\tz\in C^\infty(M)$, $\tu\in\Lambda^1(M)$, $\td\in\Lambda^2_{14}(M)= \{\beta\in\Lambda^2(M) : *_\f(\f\W\beta) = -\beta\}$, 
$\ttr\in\Lambda^3_{27}(M)= \{\rho\in\Lambda^3(M) : \f\W\rho = 0 \mbox{\ and} *_\f\f\W\rho = 0\}$ are the {\it intrinsic torsion forms} of the $G_2$-structure. 
Also in this case it is possible to classify the $G_2$-structures in terms of the non-vanishing torsion forms. For example, if $\f$ is both closed and co-closed, then all the torsion forms 
vanish, ${\rm Hol}(g_\f)\subseteq G_2$ and the $G_2$-structure is called {\it parallel}. If $\f$ is a closed form, then all the torsion forms but $\td$ vanish and the $G_2$-structure is 
said to be {\it calibrated}. If the only non-vanishing torsion forms are $\tu$ and $\td$, then at least locally the metric $g_\f$ is conformally equivalent to the metric induced by a calibrated 
$G_2$-structure and the $G_2$-structure is called {\it locally conformal calibrated}. If the only vanishing torsion form is $\td$, then the $G_2$-structure is said to be {\it integrable}. 
In this case there exists a unique affine connection with totally skew-symmetric torsion preserving the $G_2$-structure by \cite{FI}.

Consider now $(\omega,\psip)$ and two smooth functions $F:I\rightarrow \C-\{0\}$ and $G:I\rightarrow \R^+$, the following 3-form defines a $G_2$-structure on $I\times N$  
(\cite{KMT})
$$\f = \Re(F^3\Psi)+G|F|^2\omega\W dt,$$
where $t$ is the coordinate on $I$. Moreover, we have
\begin{eqnarray}
g_\f & = & G^2dt^2+|F|^2h,\nonumber\\
dV_{g_\f} & = & G|F|^6 dt\W dV_h,\nonumber\\
*_{\f}\f & = & G\,\Im(F^3\Psi)\W dt + \frac12|F|^4\omega^2.\nonumber
\end{eqnarray}
For some particular choices of the interval $I$ and the functions $F$ and $G$, we obtain the following remarkable manifolds:
\begin{itemize}
\item the {\it cylinder} $Cyl(N)$ with metric  $dt^2+h$, if $I=\R$ and $G, F\equiv 1$,
\item the {\it cone} $C(N)$ with the metric $dt^2+t^2h$, if $I=\R^+$, $G\equiv 1$ and $F(t)=t$,
\item the {\it sin-cone} $SC(N)$ with the metric $dt^2+\sin^2(t)h$, if $I=(0,\pi)$, $G\equiv1$ and $F(t)=\sin(t)e^{i\frac{t}{3}}$.
\end{itemize}
Observe that with the choice $G\equiv1$, the manifold $I\times N$ with metric $dt^2+|F|^2h$ is the warped product of $I$ and $N$ with warping function $|F|$. 
Using the expression of the Ricci tensor of the warped product metric \cite{On}, it is possible to show the following general properties (see also \cite{BG}).
\begin{Proposition}\label{einsteincone}
Let $(M^m,g)$ be a Riemannian manifold of dimension $m$. Then the cone metric $dt^2+t^2g$ is Ricci-flat if and only if the metric $g$ is Einstein with ${\rm Ric}(g)= (m-1)g$.
\end{Proposition}
\begin{Proposition}\label{einsteinsincone}
Let $(M^m,g)$ be a Riemannian manifold of dimension $m$ with Einstein metric $g$ such that ${\rm Ric}(g)= (m-1)g$. Then the sin-cone metric $dt^2+\sin^2(t)h$ is Einstein 
with Einstein constant $m$.
\end{Proposition}

\section{Coupled structures and Supersymmetry}\label{cpdsusy}
In \cite{LT}, the authors considered the problem of finding necessary and sufficient conditions for $\mathcal{N} = 1$ compactification of (massive) IIA supergravity to four-dimensional 
anti-de Sitter space on manifolds endowed with an $\SU(3)$-structure. 
As a result, they obtained a set of constraints the intrinsic torsion forms of the $\SU(3)$-structure $(\omega,\psip)$ on the internal 
manifold have to satisfy, we recall them here brief{}ly. Supersymmetry equations and the Bianchi identities constrain the intrinsic torsion to lie in the space $\mathcal{W}_1^-\oplus
\mathcal{W}_2^-$, i.e., the only non-vanishing intrinsic torsion forms are $w_1^-$ and $w_2^-$. Furthermore, in absence of sources, the Bianchi identities provide a further constraint on 
the exterior derivative of $\Wd$
\begin{equation}\label{dw2proppsip}
d\Wd \propto \psip,
\end{equation}
and the norms of $w_1^-$ and $\Wd$ have to satisfy the following inequality \cite{KLT}
\begin{equation}\label{normcond}
3(w_1^-)^2 \geq  |\Wd|^2,
\end{equation}
where $|\cdot|$ denotes the norm with respect to the metric $h$ induced by the $\SU(3)$-structure.
In the massless limit, the solutions reduce to ${\rm AdS}_4\times N$, $N$ being a compact 6-manifold endowed with an $\SU(3)$-structure with torsion in 
$\mathcal{W}_1^-\oplus\mathcal{W}_2^-$ and for which \eqref{dw2proppsip} holds.
Moreover, it was observed in \cite{KLT} that the conditions \eqref{dw2proppsip} and \eqref{normcond} can be relaxed in the presence of sources. 

It is then worth studying from the mathematical point of view the properties of this kind of $\SU(3)$-structures. In what follows, we suppose that the manifold $N$ is connected.

First of all, we recall that $\SU(3)$-structures having torsion class $\mathcal{W}_1^-\oplus\mathcal{W}_2^-$ are known as {\it coupled} structures \cite{Sal} or {\it restricted half-flat} 
structures \cite{Lar} in literature. They can be defined as the subclass of half-flat structures having $w_3\equiv0$. In this case, $d\omega$ is proportional to $\psip$, the intrinsic torsion 
form $w_1^-$ is constant \cite{Raf} and has to be nonzero if we want the intrinsic torsion $\tau$ to belong to the class $\mathcal{W}_1^-\oplus\mathcal{W}_2^-$. 
Thus, if we let $c:= -\frac32w_1^- $, we have
\begin{equation}
\begin{array}{ccl}\label{hf}
d\omega &=& c\psip,\\
d\psip &=&0,\\
d\psim &=& -\frac23c\omega^2 - \Wd\W\omega.
\end{array}
\end{equation}
The 2-form $\Wd$ lies in the space $\Lambda^{1,1}_0(N)$, therefore it satisfies the following properties:
\begin{eqnarray}
\Wd\W\omega^2&=&0,\\
\Wd\W\psi_{\pm}&=&0,\\
\Wd\W\omega &=& -*\Wd.\label{starw2}
\end{eqnarray}
Using \eqref{starw2} and the expression of $d\psim$, it is easy to show that the 2-form $\Wd$ is co-closed, that is $\de\Wd = *d*\Wd = 0$.

\begin{remark}
Observe that if a manifold admits a coupled structure $(\omega,\psip)$ with coupled constant $c\in\R-\{0\}$ such that $d\omega=c\psip$, then one can choose a nonzero real constant 
$r$, define $\tilde\omega:= r^2\omega$, $\tilde\psi_+:=r^3\psip$ and obtain a new coupled structure $(\tilde\omega,\tilde\psi_+)$ with coupled constant 
$\tilde{c}=\frac{c}{r}$. In particular, it is always possible to find a coupled structure having positive coupled constant.
\end{remark}

From the results of \cite{BV}, we have that the scalar curvature of the metric $h$ induced by a coupled structure is given by
\begin{equation}\label{scalcoupled}
{\rm Scal}(h)=\frac{15}{2}(w_1^-)^2-\frac12|\Wd|^2.
\end{equation}
Moreover, the forms $\phi_1$ and $\phi_2$ appearing in the traceless part of the Ricci tensor are
\begin{equation}\renewcommand\arraystretch{1.4}\label{phicoupled}
\begin{array}{rcl}
\phi_1 &=&\frac14*(\Wd\W\Wd) +\frac14\de(w_1^-\psip),\\
\phi_2 &=& -2*J(d\Wd).
\end{array}
\end{equation}

Let us now focus on the condition \eqref{dw2proppsip}. It forces the proportionality constant between $d\Wd$ and $\psip$ to satisfy the following result.
\begin{Proposition}\label{dw2psip}
Let $(\omega,\psip)$ be a coupled $\SU(3)$-structure and suppose that $d\Wd$ is proportional to $\psip$, then it holds
\begin{equation*}
d\Wd = -\frac{|\Wd|^2}{4}\psip.
\end{equation*}
Moreover, the norm of $\Wd$ is constant.
\end{Proposition}
\begin{proof}
First of all, observe that if $d\Wd=k\psip$ for some function $k\in C^\infty(N)$, then $k$ has to be constant. Indeed, taking the exterior derivatives of both sides we get
$$dk\W\psip=0,$$
which implies $dk=0$.
Now suppose that $d\Wd=k\psip$. Then starting from $\Wd\W\psim=0$, taking the exterior derivatives of both sides and using the previous identities we have
$$\renewcommand\arraystretch{1.4}
\begin{array}{lcccl}
0 &=& d\Wd\W\psim+\Wd\W d\psim &=& k\psip\W\psim-\Wd\W\Wd\W\omega\\
   &=&\frac23k\omega^3 +\Wd\W*\Wd &=&\frac23k\omega^3 +|\Wd|^2*1\\
   &=&\frac23k\omega^3 +|\Wd|^2\frac16\omega^3. & &
\end{array}
$$
Thus $k = -\frac{|\Wd|^2}{4}.$ From the observation made at the beginning of the proof we also get that $|\Wd|$ is constant.
\end{proof}

From Proposition \ref{dw2psip} and the fact that $w_1^-$ is constant, we obtain the following constraint.
\begin{Proposition}\label{scalconst}
Let $(\omega,\psip)$ be a coupled $\SU(3)$-structure such that $d\Wd$ is proportional to $\psip$. Then the scalar curvature of the metric induced by the coupled structure is constant.
\end{Proposition}
\begin{proof} Consider the expression \eqref{scalcoupled} of the scalar curvature of $h$  and conclude using the fact that both $w_1^-$ and $|\Wd|$ are constant.\end{proof}

Consider now condition \eqref{normcond}, this implies a further constraint on the scalar curvature.
\begin{Proposition}\label{scalarpositive}
Let $(\omega,\psip)$ be a coupled $\SU(3)$-structure whose non-vanishing intrinsic torsion forms satisfy $3(w_1^-)^2 \geq |\Wd|^2$. 
Then the scalar curvature of the metric induced by the coupled structure is positive. Moreover, it is also constant if $d\Wd$ is proportional to $\psip$.
\end{Proposition}
\begin{proof}
Using the expression of the scalar curvature of a coupled structure and the inequality $3(w_1^-)^2 \geq |\Wd|^2$ we get
$${\rm Scal}(h)=\frac{15}{2}(w_1^-)^2-\frac12|\Wd|^2 \geq 2|\Wd|^2>0.$$
Moreover, if $d\Wd$ is proportional to $\psip$, then the scalar curvature is constant by Proposition \ref{scalconst}. 
\end{proof}

It is also easy to characterize the coupled structures having $d\Wd$ proportional to $\psip$ and inducing an Einstein metric:
\begin{Proposition}\label{coupledeinsteindw2}
Let $(\omega,\psip)$ be a coupled $\SU(3)$-structure such that $d\Wd$ is proportional to $\psip$. Then the induced metric $h$ is Einstein if and only if the following identity holds
$$*(\Wd \W \Wd) = w_1^-\Wd - \frac{|\Wd|^2}{3}\omega.$$
\end{Proposition}
\begin{proof}
Recall that a Riemannian metric $h$ is Einstein if and only if ${\rm Ric}^0(h) = 0$. We know that
$${\rm Ric}^0(h) = \iota^{-1}(E_1(\phi_1)) + \gamma^{-1}(E_2(\phi_2)),$$
where $\phi_1$ and $\phi_2$ for a coupled structure are given in \eqref{phicoupled}. Now, using the fact that $d\Wd$ is proportional to $\psip$, one gets that also $\phi_2$ is 
proportional to $\psip$. Thus $E_2(\phi_2)=0$, since $\psip$ belongs to a subspace of $\Lambda^3(N)$ which is disjoint from $\Lambda^{2,1}_0(N)$. Moreover, 
$$\phi_1 = \frac14*(\Wd\W\Wd) -\frac12(w_1^-)^2~\omega -\frac14w_1^-\Wd$$
and
$$E_1(\phi_1) = \frac14*(\Wd \W \Wd)  -\frac14w_1^-\Wd + \frac{1}{12}|\Wd|^2\omega.$$
Therefore, ${\rm Ric}^0(h) = \iota^{-1}(E_1(\phi_1))$ is zero if and only if $E_1(\phi_1)$ is zero, and from this the assertion follows.
\end{proof}

\section{Examples}\label{exsec} 
In this section, we examine some examples of 6-manifolds admitting an $\SU(3)$-structure satisfying (all or in part) the properties discussed in Section \ref{cpdsusy}. 

\subsection{Nilmanifolds}
We recall here the definition of a nilmanifold and some useful properties. 
\begin{Definition}
Let $G$ be a connected, simply connected, nilpotent Lie group and $\Gamma$ a cocompact discrete subgroup. The compact quotient manifold $G/\Gamma$ is called {\it nilmanifold}.
\end{Definition}
In the general case, every left invariant tensor on $G$ passes to the quotient defining an invariant tensor on the nilmanifold $G/\Gamma$. 
Moreover, all the 34 six-dimensional nilpotent Lie algebras existing up to isomorphisms \cite{Mag} satisfy the following result
\begin{Proposition}[\cite{Mal}]
Let $\mathfrak{g}$ be a nilpotent Lie algebra and suppose there exists a basis of it such that the structure constants determined with respect to this basis are rational numbers. 
Then, denoted by $G$ the simply connected nilpotent Lie group whose Lie algebra is $\mathfrak{g}$, there exists a discrete subgroup $\Gamma$ of $G$ such that $G/\Gamma$ is a 
nilmanifold.
\end{Proposition}
It then follows that there is a $1-1$ correspondence between invariant $\SU(3)$-structures $(\omega,\psip)$ on a nilmanifold and pairs $(\omega,\psip)$ defining an $\SU(3)$-structure 
on its nilpotent Lie algebra. This allows to work only with $\SU(3)$-structures defined on nilpotent Lie algebras. 

Since every nilpotent Lie group is solvable, the following result by Milnor holds in the case we are considering.
\begin{Theorem}[\cite{Mil}]
Let $G$ be a solvable Lie group. Then every left invariant metric on $G$ is either flat or has strictly negative scalar curvature.
\end{Theorem}
In particular, if a nilpotent Lie algebra is endowed with an inner product $h$, then ${\rm Scal}(h)$ is non-positive. 
As a consequence, using Proposition \ref{scalarpositive} it is immediate to show the
\begin{Proposition}
There are no six-dimensional nilmanifolds admitting an invariant coupled structure satisfying the condition $3(w_1^-)^2 \geq  |\Wd|^2.$
\end{Proposition}

Thus, we can only look for nilpotent Lie algebras endowed with a coupled structure $(\omega,\psip)$ having $d\Wd$ proportional to $\psip$. 
In \cite{FR}, we showed that among the 34 non-isomorphic six-dimensional nilpotent Lie algebras there are only two of them admitting a coupled structure, 
we recall the result here.
\begin{Proposition}\label{nilpclass}
Let $\mathfrak{g}$ be a six-dimensional, non-abelian, nilpotent Lie algebra endowed with a coupled $\SU(3)$-structure $(\omega,\psip)$. Then $\mathfrak{g}$ 
is isomorphic to one of the following nilpotent Lie algebras
\begin{eqnarray}
\mathfrak{I} &=& (0,0,0,0,e^{14}+e^{23},e^{13}-e^{24}),\\
\mathfrak{N} &=& \left(0,0,0,e^{13},e^{14}+e^{23},e^{13}-e^{15}-e^{24}\right).
\end{eqnarray}
\end{Proposition}
\begin{remark}
Recall that the notation $\mathfrak{I} = (0,0,0,0,e^{14}+e^{23},e^{13}-e^{24})$ means that there exists a basis of 1-forms $(e^1,\ldots,e^6)$ for $\mathfrak{I}^*$ such that 
$de^1=0,de^2=0,de^3=0,de^4=0, de^5=e^{14}+e^{23},de^6=e^{13}-e^{24}$, where $d$ is the Chevalley-Eilenberg differential.
\end{remark}

Observe that the two Lie algebras $\mathfrak{I}$ and $\mathfrak{N}$ are isomorphic respectively to the Lie algebras labelled by $\mathfrak{n}_{28}$ and $\mathfrak{n}_{9}$ 
in the work \cite{FR}. Here they are given with different structure equations since in both cases the frame $(e^1,\ldots,e^6)$ is an adapted frame for the coupled 
$\SU(3)$-structure. We emphasize some properties of these coupled structures in the following examples.
\begin{Example}\label{iwacoupled}
$\mathfrak{I}=(0,0,0,0,e^{14}+e^{23},e^{13}-e^{24})$ is (isomorphic to) the well known Iwasawa Lie algebra, which is the Lie algebra of the six-dimensional nilmanifold known 
in literature as {\it Iwasawa manifold} (see for instance \cite{AGS} for the definition).
Since the frame $(e^1,\ldots,e^6)$ is adapted, we have that the pair
$$
\begin{array}{lcl}
\omega &=& e^{12}+e^{34}+e^{56},\\
\psip &=& e^{135}-e^{146}-e^{236}-e^{245},
\end{array}
$$
defines an $\SU(3)$-structure on $\mathfrak{I}$. 
In this case $d\omega=-\psip$ and the non-vanishing intrinsic torsion forms are:
$$\renewcommand\arraystretch{1.4}
\begin{array}{lcl}
w_1^- &=& \frac23,\\
\Wd &=& -\frac43e^{12}-\frac43e^{34}+\frac83e^{56}.
\end{array}
$$
It is easy to check that condition \eqref{dw2proppsip} is satisfied
$$d\Wd = -\frac83\psip$$
and that $-\frac14|\Wd|^2=-\frac83,$ as we expected from Proposition \ref{dw2psip}.  
Finally, the scalar curvature of the metric $h$ induced by the coupled structure is ${\rm Scal}(h)=-2.$
\end{Example}

\begin{Example}
Consider the Lie algebra $\mathfrak{N} = \left(0,0,0,e^{13},e^{14}+e^{23},e^{13}-e^{15}-e^{24}\right)$. Since the frame $(e^1,\ldots,e^6)$ is adapted, we have that the pair
$$
\begin{array}{lcl}
\omega &=& e^{12}+e^{34}+e^{56},\\
\psip &=& e^{135}-e^{146}-e^{236}-e^{245},
\end{array}
$$
defines an $\SU(3)$-structure on $\mathfrak{N}$. Moreover, $d\omega=-\psip$ and the non-vanishing intrinsic torsion forms are:
$$\renewcommand\arraystretch{1.4}
\begin{array}{lcl}
w_1^- &=& \frac23,\\
\Wd &=& -\frac43e^{12}-\frac43e^{34}+e^{36}-e^{45}+\frac83e^{56}.
\end{array}
$$
In this case $d\Wd$ is not proportional to $\psip$ and the scalar curvature of the metric $h$ induced by the coupled structure is ${\rm Scal}(h)=-3.$
\end{Example}

The fact that the Iwasawa manifold admits an invariant coupled structure was also observed in \cite{LT}, 
where the authors wrote it was the unique nilmanifold admitting a coupled structure they knew. Proposition \ref{nilpclass} states that, up to isomorphisms, 
there are only two non-abelian nilpotent Lie algebras admitting a coupled structure, one of which is the Iwasawa. Moreover, as observed in Example \ref{iwacoupled}, 
the coupled structure on the Iwasawa Lie algebra satisfies condition \eqref{dw2proppsip}, 
i.e., $d\Wd$ is proportional to $\psip$. Thus, it is a natural question to ask whether $\mathfrak{N}$ admits a coupled structure satisfying \eqref{dw2proppsip} or not.
In \cite{CKKLTZ}, the authors looked for the possible nilmanifolds admitting an invariant coupled structure satisfying \eqref{dw2proppsip} and concluded that a systematic scan of all 
the possible six-dimensional nilmanifolds yields to two possibilities: the six-torus and the Iwasawa manifold. 
The six-torus has abelian Lie algebra, so it is not considered in Proposition \ref{nilpclass}. Moreover, the intrinsic torsion forms of any invariant $\SU(3)$-structure defined on it are 
all zero. Anyway, this result seems to answer negatively our question and we can prove this is actually what happens.
\begin{Proposition}
There are no coupled $\SU(3)$-structures on $\mathfrak{N}$ for which the exterior derivative of the intrinsic torsion form $\Wd$ is proportional to $\psip$.
\end{Proposition}
\begin{proof}
The idea is to describe all the possible coupled structures on $\mathfrak{N}$ and see if there exists one of these whose intrinsic torsion form $\Wd$ satisfies the required condition.  
Let us start considering a generic 2-form $\omega$ on $\mathfrak{N}$, we can write it as
$$
\omega = \sum_{1\leq i<j\leq6} \omega_{ij}e^{ij},
$$
where $\omega_{ij}$ are real numbers. We may think the 15-tuple $(\omega_{12},\ldots,\omega_{56})=:(\omega_{ij})$ as a point in the affine space $\mathbb{A}_\R^{15}-\{0\}$.
The homogeneous polynomial $P$ of degree 3 in the unknowns $\omega_{ij}$ appearing as coefficient of $e^{123456}$ 
in the expression of $\omega^3$ has to be non-vanishing, this gives a first constraint for $(\omega_{ij})$. 
Since we want a coupled structure, we consider a 3-form $\psip$ on $\mathfrak{N}$ given by $\psip = cd\omega,$ for some nonzero real number $c$. Assuming
$$
\lambda(\psip) = -4c^4\omega_{56}^2(\omega_{36}\omega_{56}-\omega_{45}\omega_{56}-\omega_{46}^2+\omega_{56}^2)<0,
$$
that is $\omega_{56}\neq0$ and $B:= \omega_{36}\omega_{56}-\omega_{45}\omega_{56}-\omega_{46}^{2}+\omega_{56}^{2}>0$, we can compute the almost complex structure $J$ induced by the stable form $\psip$. Now, we change the basis from $(e_1,\ldots,e_6)$ to a basis 
$(E_1,\ldots,E_6)$ which is adapted for $J$. To do this, it suffices to define $E_i=e_i$ and $E_{i+1}=Je_i$ for $i=1,3,5$. With respect to $(E_1,\ldots,E_6)$, the matrix associated to $J$ 
is skew-symmetric with non-vanishing entries given by ${J^2}_1=1={J^4}_3 = {J^6}_5$.
We can then compute the new structure equations with respect to the dual basis $(E^1,\ldots,E^6)$, obtaining
\begin{eqnarray}
dE^i &=& 0, \quad i=1,2,3\nonumber\\
dE^4 &=&\frac{\omega_{56}} {{\sqrt{B}}}E^{13}, \nonumber\\
dE^5 &=& -\frac{\omega_{46}}{\omega_{56}}E^{13} +\frac{\sqrt{B}}{\omega_{56}} 
\left(E^{14}+E^{23}\right)\nonumber,\\
dE^6 &=&  -\frac{\omega_{26}}{\omega_{56}}E^{12} -\frac{\omega_{46}}{\omega_{56}}E^{14} -\frac{\omega_{36}\omega_{56}+\omega_{45}\omega_{56}-\omega_{46}^{2}-\omega_{56}
^{2}}{\omega_{56}\sqrt {B}}E^{13} -\frac{\omega_{56}} {{\sqrt{B}}}E^{15}  -\frac{\sqrt {B}}{\omega_{56}}E^{24}.\nonumber
\end{eqnarray}
Moreover, we have
\begin{eqnarray}
\psip &=& -c\frac{B}{\omega_{56}}\left(E^{135}-E^{146}-E^{236}-E^{245}\right), \nonumber\\
\psim&=& -c\frac{B}{\omega_{56}}\left(E^{136}+E^{145}+E^{235}-E^{246}\right).\nonumber
\end{eqnarray}
We can write $\omega$ with respect to the new basis and impose it is of type $(1,1)$ with respect to $J$, obtaining 3 equations in the variables $\omega_{ij}$ 
which can be solved under the 
constraint $\lambda(\psip)<0$. We can then consider the symmetric matrix $H$ associated to $h(\cdot,\cdot)=\omega(\cdot,J\cdot)$ with respect to the basis $(E_1,\ldots,E_6)$ and 
denote by $\mathcal{P}\subset \mathbb{A}^{15}_\R$ the set on which it is positive definite. One can check that $P\neq0$ when $(\omega_{ij})\in \mathcal{P}$.
Now, if we let $(\omega_{ij})$ vary in the (non-empty) set $\mathcal{Q}:=\mathcal{P}\cap\{(\omega_{ij}) : \lambda(\psip)<0\}$, we have all the possible non-normalized coupled $
\SU(3)$-structures on 
$\mathfrak{N}$. The intrinsic torsion form $w_1^-$ is always $-\frac{2}{3c}$, while $\Wd$ can be 
computed from its defining properties and the expression of $d\psim$. We are interested in the coupled structures having $\Wd$ such that $d\Wd$ is proportional to $\psip$. 
Thus, we can start with a generic 2-form $w$ of type $(1,1)$ with respect to $J$ and write it as
\begin{eqnarray}
w&=&w_{12}E^{12}+w_{34}E^{34}+w_{56}E^{56}+w_{13}(E^{13}+E^{24}) + w_{14}(E^{14}-E^{23})+w_{15}(E^{15}+E^{26})\nonumber\\
   & & w_{16}(E^{16}-E^{25})+w_{35}(E^{35}+E^{46})+w_{36}(E^{36}-E^{45}),\nonumber
\end{eqnarray}
where $w_{ij}$ are real numbers. Then, we have to impose that $w$ is primitive ($w\W\omega^2=0$) and fulfills
$$d\psim=-\frac{2}{3c}\omega^2-w\W\omega$$
and that $dw$ is proportional to $\psip$. 
The last condition gives rise to a set of polynomial equations in the variables $w_{ij}$ with coefficients depending on $\omega_{ij}$ which can be solved in $\mathcal{Q}$. 
The condition on $d\psim$ gives 13 equations of the same kind as before, we can solve 4 of them, namely those obtained comparing the coefficients of
$E^{3456},E^{2356},E^{1256},E^{2345}$, but then we get that some of the remaining equations can be solved only if $c=0$ or $\lambda(\psip)=0$. The assertion is then proved.  
\end{proof}

The previous results can be summarized as follows
\begin{Proposition}
Let $\mathfrak{g}$ be a six-dimensional, non-abelian, nilpotent Lie algebra endowed with a coupled $\SU(3)$-structure $(\omega,\psip)$ having $d\Wd$ proportional to $\psip$. 
Then $\mathfrak{g}$ is isomorphic to the Iwasawa Lie algebra.
\end{Proposition}

\subsection{Twistor spaces}
In the work \cite{Tom}, it was observed that on the twistor space $Z$ over a self-dual Einstein 4-manifold $(M^4,g)$ there exists a coupled structure. Moreover, for a suitable value of the 
scalar curvature of $g$, the metric induced by this structure is Einstein. 

Recall that given a four-dimensional, oriented Riemannian manifold $(M^4,g)$, the set of positive, orthogonal almost complex structures on $M^4$ forms a smooth manifold $Z$ 
called the {\bf twistor space} of $M^4$, which can be viewed as the 2-sphere bundle $\pi:Z\rightarrow M^4$ consisting of the unit $-1$ eigenvectors 
of the Hodge operator acting on $\Lambda^2TM^4$ \cite{Mus}.

On $Z$, it is possible to define two almost complex structures (see for example \cite{AGI}), one of which is never integrable as shown in \cite{ES}. 
Let us denote it by $J$.

When the metric $g$ is self-dual and Einstein, Xu showed in \cite{Xu} that on $(Z,J)$ there exists a basis of $(1,0)$-forms $\e^1,\e^2,\e^3$ such that 
the first structure equations are:
\begin{equation}
d \left(\begin{array}{c} \e^1 \\ \e^2 \\ \e^3 \end{array} \right) = -\left(\begin{array}{cc} \alpha & 0 \\ 0 & -{\rm tr}(\alpha) \end{array} \right) \W \left(\begin{array}{c} \e^1 \\ \e^2 \\ \e^3 
\end{array} \right) + \left(\begin{array}{c} \overline{\e^2\W\e^3} \\ \overline{\e^3\W\e^1} \\ \sigma\overline{\e^1\W\e^2} \end{array} \right), 
\end{equation}
where $\alpha$ is a $2\times2$ skew-Hermitian matrix of 1-forms and $\sigma:=\frac{{\rm Scal}(g)}{24}$. 
Using these, it is easy to show that the following pair of forms defines a coupled $\SU(3)$-structure on $Z$ \cite{Tom}
$$\renewcommand\arraystretch{1.4}
\begin{array}{rcl}
\omega &=& \frac{i}{2}\left(\e^1\W\overline{\e^1}+\e^2\W\overline{\e^2}+\e^3\W\overline{\e^3}\right), \\
\Psi &=& i(\e^1\W\e^2\W\e^3).
\end{array}
$$
Observe that $J$ is the almost complex structure induced by $\Re(\Psi)$ and that the metric induced by $\omega$ and $J$ takes the following form
$$h = \e^1\odot\overline{\e^1}+\e^2\odot\overline{\e^2}+\e^3\odot\overline{\e^3}.$$
Moreover, the non-vanishing intrinsic torsion forms have the following expressions
$$\renewcommand\arraystretch{1.4}
\begin{array}{rcl}
w_1^- &=& \frac23(\sigma+2),\\
w_2^- &=& -\frac23i(\sigma-1)\left(\e^1\W\overline{\e^1}+\e^2\W\overline{\e^2}-2\e^3\W\overline{\e^3}\right),
\end{array}
$$
$d\Wd$ is proportional to $\psip$
$$d\Wd = - \frac83(\sigma-1)^2\psip,$$
and $3(w_1^-)^2\geq|\Wd|^2$ if and only if $\frac{10-6\sqrt{2}}{7}\leq\sigma\leq\frac{10+6\sqrt{2}}{7}$.

We can consider a local frame $(e^1,\ldots,e^6)$ for $\Lambda^1(Z)$ such that $\e^1=e^1+ie^2$, $\e^2=e^3+ie^4$, $\e^3=e^5+ie^6$ and compute the Ricci curvature 
of the metric induced by the coupled structure using the results of \cite{BV}. What we get is that the scalar curvature of $h$ is
$$
{\rm Scal}(h) = -2\sigma^2+24\sigma+8
$$
and the traceless part of the Ricci tensor of $h$ with respect to the considered frame has the following form 
$${\rm Ric}^0(h) = -\frac23(\sigma-1)(\sigma-2){\rm diag}(1,1,1,1,-2,-2).$$
Thus, the metric $h$ is Einstein if and only if $\sigma=1$ or $\sigma=2$, that is if and only if the scalar curvature of $g$ is $24$ or $48$ respectively. In the first case the coupled 
structure is actually nearly K\"ahler since $\Wd=0$, while in the second case we get an example of a coupled $\SU(3)$-structure inducing an Einstein metric. More in detail, the latter 
has the following non-vanishing intrinsic torsion forms
$$
\renewcommand\arraystretch{1.2}
\begin{array}{rcl}
w_1^- &=& \frac83,\\
\Wd &=& -\frac43\left(e^{12}+e^{34}-2e^{56}\right).
\end{array}
$$
In particular, the coupled constant is $c=-4$ and the scalar curvature is ${\rm Scal}(h) = 48$. 
Moreover, the characterization given in Proposition \ref{coupledeinsteindw2} is satisfied by this example.

Recall that when ${\rm Scal}(g)>0$, a compact, self-dual, Einstein 4-manifold $(M^4,g)$ is isometric either to $S^4$ or to $\C\mathbb{P}^2$ with their canonical 
metrics \cite[Thm. 13.30]{Bes}, thus $Z$ is either $\C\mathbb{P}^3$ or the flag manifold $\SU(3)/T^2$.

\subsection{$G_2$-structures with special metrics induced by coupled Einstein structures}
We can now use the coupled Einstein structure on $Z$ to construct a $G_2$-structure with full intrinsic torsion inducing an Einstein metric 
and a locally conformal calibrated $G_2$-structure inducing a Ricci-flat metric.

First of all, we rescale the coupled Einstein structure on $Z$ it in the following way
$$
\begin{array}{rcl}
\tilde\omega &=& \frac85\omega,\\
\tilde\psi_+ &=& \left(\frac85\right)^{\frac32}\psip.
\end{array}
$$
Then, $(\tilde\omega,\tilde\psi_+)$  is a coupled structure with coupled constant $c=-\sqrt{10}$ 
and inducing the metric $\tilde{h} = \frac85h$. Moreover, ${\rm Scal}(\tilde{h})=30$ and ${\rm Ric}(\tilde{h}) = 5\tilde{h}$.

As we observed in Section \ref{preliminaries}, starting from the coupled Einstein structure $(\tilde\omega,\tilde\psi_+)$, we can construct a $G_2$-structure $\f$ on the sin-cone $S(Z)$ 
inducing the sin-cone metric 
$g_\f=dt^2+\sin^2(t)\tilde{h}$. By Proposition \ref{einsteinsincone}, we then have that $g_\f$ is Einstein with Einstein constant 6. 
Moreover, it is not difficult to show that the intrinsic torsion forms of the $G_2$-structure induced on the sin-cone by a coupled structure with coupled constant $c$ are
$$\renewcommand\arraystretch{1.4}
\begin{array}{rcl}
\tz &=& \frac{8c+4}{7},\\
\tu &=& \left(1-\frac{c}{3}\right)\cot(t)dt,\\
\td &=& -\frac{\sin(2t)}{2}\Wd,\\
\ttr &=&\frac{c-3}{7}\left(\sin^4(t)\psim-\sin^3(t)\cos(t)\psip+\frac43\sin^2(t)dt\W\omega\right)-\sin^2(t)dt\W\Wd.
\end{array}
$$ 
Thus, the coupled Einstein structure $(\tilde\omega,\tilde\psi_+)$ induces a $G_2$-structure with full intrinsic torsion and Einstein metric on the sin-cone $S(Z)$.

If we consider the $G_2$-structure $\f$ induced on the cone $C(Z)$ by $(\tilde\omega,\tilde\psi_+)$, then the metric $g_\f = dt^2+t^2\tilde{h}$ is Ricci-flat by Proposition 
\ref{einsteincone}. Moreover, the non-vanishing intrinsic torsion forms of the $G_2$-structure constructed on the cone from a coupled structure with coupled constant $c$ are 
$$
\begin{array}{rcl}
\tu &=& \left(\frac1t-\frac{1}{3t}c\right)dt,\\
\td &=& -t\Wd.
\end{array}
$$
Therefore, the coupled Einstein structure $(\tilde\omega,\tilde\psi_+)$ induces a locally conformal calibrated $G_2$-structure on the cone $C(Z)$ whose associated metric is Ricci-flat. 

\begin{remark}
It is worth observing here that calibrated $G_2$-structures inducing a Ricci-flat metric are actually parallel \cite{Br}. 
The previous example shows that a result of this kind is not true anymore for locally conformal calibrated $G_2$-structures.
\end{remark}

\section{Flows}\label{flows}
In this section, we study the behaviour of coupled structures with respect to known evolution equations (flows) of $\SU(3)$-structures. 

The {\it Hitchin flow}, introduced in \cite{Hit} as the Hamiltonian flow of a certain functional, allows to construct (non-complete) metrics with holonomy in $G_2$ starting from a suitable 
$\SU(3)$-structure. The idea is to consider a 6-manifold $N$ endowed with an $\SU(3)$-structure $(\omega,\psip)$ and define a $G_2$-structure on $M:=I\times N$ 
for some interval $I\subseteq\R$ by
$$
\f = dt\W\omega+\psip,
$$
where $\omega$ and $\psip$ depend on the coordinate $t$ on $I$. If we require the $G_2$-structure to be parallel, we get that for each $t$ fixed the $\SU(3)$-structure has to 
be half-flat and that, when $t$ is not fixed, the following evolution equations have to hold 
\begin{equation}\label{hitchinflow}
\begin{cases}
\ddt\psip = d\omega\\
\ddt\omega\W\omega = -d\psim
\end{cases}.
\end{equation}
These equations are the so called Hitchin flow equations. A solution of them with initial condition a given $\SU(3)$-structure $(\omega(0), \psip(0))$ exists when the latter is 
half-flat and analytic, but may not exist when the analytic hypothesis is dropped \cite{Br2}. Moreover, it is easy to show that an $\SU(3)$-structure $(\omega(t), \psip(t))$ which is half-flat 
for $t=0$ and evolves as prescribed in \eqref{hitchinflow} stays half-flat as long as it exists. 

In the work \cite{DLS}, a generalization of the Hitchin flow was used to study the moduli space of $\SU(3)$-structure manifolds. 
The starting point to define this flow is to consider the embedding of an $\SU(3)$-structure into a non-compact manifold endowed with an integrable $G_2$-structure.
This is motivated by the subject the authors are interested in, namely four-dimensional domain wall solutions of heterotic string theory that preserve $\mathcal{N} = \frac12$ 
supersymmetry (see also \cite{GLL}). In this case, the internal six-dimensional manifold is endowed with an $\SU(3)$-structure and one can combine it with the direction perpendicular to 
the domain wall in the four-dimensional non-compact space time to get a seven-dimensional non-compact manifold endowed with a $G_2$-structure. 
The physical setting provides further constraints on the intrinsic torsion forms of the $G_2$-structure, which we will recall later. 
One can then study under which conditions a certain class of $\SU(3)$-structures is preserved by this generalized flow.

It is then a natural question to ask whether the coupled condition is preserved by the Hitchin flow and, more in general, which constraints arise requiring that a solution to the 
generalized Hitchin flow is coupled as long as it exists. We begin giving the following definition.
\begin{Definition}
Let $(\omega(t),\psip(t))$ be a solution for the Hitchin flow defined on an interval $I\subseteq\R$ containing $0$ and starting from a coupled structure at $t=0$.  
If $(\omega(t),\psip(t))$ is a coupled structure for each $t\in I$, that is $d\omega(t) = c(t)\psip(t)$ for some smooth function $c:I\rightarrow\R$, 
we call it a {\it coupled solution} for the Hitchin flow.
\end{Definition}

Coupled solutions for the Hitchin flow can be easily characterized and induce an almost complex structure not depending on $t$.
\begin{Proposition}\label{shapecpdsol}
Let $N$ be a six-dimensional manifold and suppose there exists on it a solution $(\omega(t),\psip(t))$ for the Hitchin flow starting from a coupled structure  
$(\omega(0),\psip(0))$ and defined on some interval $I\subseteq\R$ containing 0. 
If $(\omega(t),\psip(t))$ is a coupled solution, then there exists a smooth function $f:I\rightarrow\R$ such that
$$\psip(t) = f(t)\psip(0).$$
Conversely, if the pair $(\omega(t),\psip(t))$ is a solution for the Hitchin flow with $\psip(t) = f(t)\psip(0)$, then it is a coupled solution.
\end{Proposition}
\begin{proof}
If $(\omega(t),\psip(t))$ is a solution for the Hitchin flow with $\psip(t) = f(t)\psip(0)$ then from the flow equation $\ddt\psip(t) = d\omega(t)$ we obtain
$$d\omega(t) = \ddt\left( f(t)\psip(0)\right) = \left(\dt f(t)\right)\psip(0).$$
Thus the solution is a coupled structure with $c(t) = \dt f(t)$.
Suppose now that the solution is coupled,  $d\omega(t) = c(t)\psip(t)$. Then from the flow equation we obtain
$$\ddt\psip(t) = c(t)\psip(t).$$
Working in coordinates on $N$, it is easy to show that
$$\psip(t) = f(t)\psip(0),$$
where 
$$f(t) = e^{\int_0^tc(s)ds}.$$
\end{proof}

\begin{Corollary}
Let $(\omega(t),\psip(t))$ be a coupled solution for the Hitchin flow on a six-dimensional manifold $N$. 
Then the associated almost complex structure is $J(t)=J(0)$, that is, it does not depend on $t$.
\end{Corollary}
\begin{proof}
We know that $\psip(t) = f(t)\psip(0)$, therefore
$$J(t) = J_{\psip(t)} = J_{f(t)\psip(0)} = J_{\psip(0)}=J(0),$$
since the almost complex structure induced by $\psip$ does not change if we rescale $\psip$ by a real constant. 
\end{proof}

\subsection{Coupled solutions on six-dimensional nilpotent Lie algebras}
Working on six-dimensional nilpotent Lie algebras, it is possible to show that a coupled solution for the Hitchin flow may not exist. 
As we recalled in Proposition \ref{nilpclass}, the only six-dimensional nilpotent Lie algebras admitting a coupled structure are, up to isomorphisms, $\mathfrak{I}$ and $\mathfrak{N}$. 
In each case, with respect to the frame we considered, the pair $(\omega,\psip)$, where 
\begin{equation}\label{startfoa}
\begin{array}{rcl}
\omega & = & e^{12} + e^{34} + e^{56},\\
\psip     & = & e^{135}-e^{146}-e^{236}-e^{245},
\end{array}
\end{equation}
is a coupled structure with $d\omega = -\psip.$
For completeness, we observe also that  
$$\psim = J\psip = e^{136}+e^{145}+e^{235}-e^{246}.$$

The following result shows our claim.
\begin{Proposition}\label{HFnosol}
Consider the Hitchin flow on the six-dimensional nilpotent Lie algebras $\mathfrak{N}$ and $\mathfrak{I}$. Then on $\mathfrak{I}$ there exists a coupled solution starting from 
\eqref{startfoa} at $t=0$, while on $\mathfrak{N}$ there are no coupled solutions for the Hitchin flow starting from \eqref{startfoa}. 
\end{Proposition}
\begin{proof}
Let us start with $\mathfrak{I}$, it admits a coupled solution for the Hitchin flow already described in \cite{CF}. We recover it here starting from a suitable pair $(\omega(t),\psip(t))$ 
and requiring it satisfies the Hitchin flow equations.
From Proposition \ref{shapecpdsol}, we know that $(\omega(t),\psip(t))$ is a coupled solution if and only if
$$\psip(t) = f(t)\psip(0) = f(t)(e^{135}-e^{146}-e^{236}-e^{245}),$$
with $f(0) = 1$. It is also clear that $\psim(t) = f(t)\left(e^{136}+e^{145}+e^{235}-e^{246}\right)$. 
Moreover, since we already know the form of the solution, we consider three smooth functions $a_1(t), a_2(t), a_3(t)$ with $a_i(0)=1$ and such that
$$\omega(t) = a_1(t)e^{12} + a_2(t)e^{34} + a_3(t)e^{56}.$$
From now on we omit the  $t$-dependence of the considered functions. The forms $\omega(t)$ and $\psi_\pm(t)$ are compatible for each $t$ and from the normalization condition we 
get
\begin{equation}
f^2 = a_1a_2a_3. \label{IWHFN}
\end{equation}
From the first Hitchin flow equation in \eqref{hitchinflow} we obtain
\begin{equation}
\dt{f} = -a_3, \label{IWHF11}
\end{equation}
while from the second one we have
\begin{eqnarray}
\dt(a_1a_3) & = & 0, \label{IWHF21}\\
\dt(a_2a_3) & = & 0, \label{IWHF22}\\
\dt(a_1a_2) & = & -4f. \label{IWHF23}
\end{eqnarray}
From \eqref{IWHF21}, \eqref{IWHF22} and the starting conditions at $t=0$ we deduce that 
$$a_1 = a_2 = \frac{1}{a_3}.$$
Using this result and \eqref{IWHFN}, it holds necessarily
$$f = \frac{1}{\sqrt{a_3}}.$$
Thus the ODE \eqref{IWHF11} becomes
$$\dt a_3 = 2a_3^2\sqrt{a_3}$$
and solving this we get
$$a_3 = (1-3t)^{-\frac23}.$$
It is then easy to check that also \eqref{IWHF23} is satisfied.
Then the pair 
\begin{eqnarray}
\omega(t) & = & (1-3t)^{\frac23}e^{12} + (1-3t)^{\frac23}e^{34} + (1-3t)^{-\frac23}e^{56},\nonumber \\
\psip(t) & = & (1-3t)^{\frac13}(e^{135}-e^{146}-e^{236}-e^{245})\nonumber
\end{eqnarray}
is a coupled solution for the Hitchin flow.

We consider now $\mathfrak{N}$, we will show that there are no coupled solutions starting from \eqref{startfoa}. Also in this case we need 
$$\psip(t) = f(t)\psip(0) = f(t)(e^{135}-e^{146}-e^{236}-e^{245}),$$
with $f(0)=1$, while we consider 15 smooth real valued functions $b_{ij}(t)$, $1\leq i < j \leq 6$, such that 
$$\omega(t) = \sum_{1\leq i < j \leq 6}b_{ij}(t)e^{ij},$$
$b_{12}(0)=b_{34}(0)=b_{56}(0)=1$ and $b_{ij}(0)=0$ for the remaining functions. 
First of all, we impose that the equations resulting from the compatibility condition $\omega(t)\W\psip(t)=0$ are satisfied, then we consider the first and the second Hitchin flow equations 
and we compute the ODEs deriving from them. What we obtain after solving some of these dif{}ferential equations is that $f\equiv0$, which can not be possible.
\end{proof}

\subsection{Generalized Hitchin flow}
Since coupled solutions for the Hitchin flow may not exist in general as Proposition \ref{HFnosol} states, we can consider the generalized Hitchin flow and investigate which properties 
the intrinsic torsion forms have to satisfy in order to preserve the coupled condition. 

In this case, we start with an $\SU(3)$-structure $(\omega,\psip)$ depending on a parameter $t\in I\subseteq\R$ and we construct a $G_2$-structure on $M := I\times N$ by
$$
\f = \nu_t dt\W\omega+\Re(F\Psi),
$$
where $\nu_t\in C^\infty(M)$ and $F$ is a complex valued smooth function defined on $M$ and having constant module 1. 
Observe that the Riemannian metric defined by $\f$ is
$$
g_\f = \nu_t^2dt^2+h.
$$

As we already recalled, in the case of $\mathcal{N}=\frac12$ domain wall solutions the non-vanishing intrinsic torsion forms of the $G_2$-structure are $\tz,\tu,\ttr$. 
On $M=I\times N$, $\tu$ and $\ttr$ can be decomposed as
$$
\begin{array}{rcl}
\tu &=& \tut dt+\tu^N,\\
\ttr &=& dt\W \ttrt+\ttr^N,
\end{array}
$$
where $\tut$ is a smooth function on $M$, $\tu^N$ is a 1-form on $N$, $\ttrt$ is a 2-form on $N$ depending on $t$ and $\ttr^N$ is a 3-form on $N$.
Moreover, the following constraints hold
\begin{eqnarray}
\tut &=&\frac12\ddt\phi,\nonumber\\
\tu^N&=&\frac12d\phi,\nonumber\\
d_7\tz&=&0,\nonumber
\end{eqnarray}
where $\phi$ is the ten-dimensional dilaton, $d_7$ denotes the exterior derivative on $M$ and $d$ denotes it on $N$.

A general argument allows to write down the equations of the $\SU(3)$-structure flow associated to the embedding and some relations between the torsion forms of the 
$\SU(3)$-structure and the $G_2$-structure. In particular
$$
w_4 = 2\tu^N,
$$
therefore, if we have an $\SU(3)$-structure with vanishing $w_4$, we get $d\phi =2\tu^N=0.$

Following \cite{DLS}, we work in the gauge $F=1$, in this case
$$\f = \nu_t dt\W\omega+\psip.$$
If we suppose that the structure is coupled for each $t$, i.e., 
\begin{equation}
\begin{array}{ccl}\label{hf}
d\omega(t) &=& c(t)\psip(t),\\
d\psip(t) &=&0,\\
d\psim(t) &=& -\frac23c(t)(\omega(t))^2 - \Wd(t)\W\omega(t),
\end{array}
\end{equation}
where $c:I\rightarrow\R$ is a smooth function such that $w_1^-(t)=-\frac23c(t)$, then the 2-form $\omega(t)$ evolves as
\begin{equation}\label{flowomega}
\ddt\omega(t) = \lt\omega(t)+h_t,
\end{equation}
where 
\begin{eqnarray}
\lt&=&2\tut-\nu_t w_1^-(t),\label{lambdat}\\
h_t&=& \nu_t\Wd(t)-*(d\nu_t\W*\psip(t)).\label{ht}
\end{eqnarray}
Moreover, it follows from a general argument involving the flow equations that 
$$
d\lt=0
$$ 
and using one of the constraints recalled earlier we get
$$
d\tut=\frac12d\left(\ddt\phi\right)=\frac12 \ddt(d\phi)=0.
$$ 
Taking the exterior derivative of both sides of \eqref{lambdat} we then have
$$d\nu_t=0,$$
thus, $\nu_t$ is actually a function of $t$ and \eqref{ht} becomes $h_t= \nu_t\Wd.$

\begin{remark}
With our convention, $\Wd$ here is $-\Wd$ in the work \cite{DLS}.
\end{remark}

The flow equations for $\psip(t)$ and $\psim(t)$ determined in \cite{DLS} reduce to the following in the coupled case:
\begin{eqnarray}
\ddt\psip(t) &=& \frac32\lt\psip(t)-\frac74\tz \nu_t\psim(t)-\nu_t\gamma,\label{flowpsip}\\
\ddt\psim(t)&=& \frac74\tz \nu_t\psip(t)+\frac32\lt\psim(t)+\nu_t J\gamma,\label{flowpsim}
\end{eqnarray}
where $\gamma$ is a primitive 3-form of type $(2,1)+(1,2)$ appearing in the expression of the Hodge dual of $\ttr^N$ on $N$.

We derive now all the conditions that arise requiring these flow equations preserve the coupled condition. We may sometimes omit the $t$-dependence of the forms for brevity.

First of all, suppose that for each $t$ the coupled condition $d\omega(t) = c(t)\psip(t)$ holds. Differentiating both sides with respect to $t$, we have
$$
d\left(\ddt\omega\right) = \dot{c}\psip+c\left( \frac32\lt\psip-\frac74\tz \nu_t\psim-\nu_t\gamma\right).
$$
Moreover, taking the exterior derivative of both sides of \eqref{flowomega}, using $d\nu_t=0$ and the hypothesis on the coupled condition,  we obtain
$$
d\left(\ddt\omega\right) =\lt c\psip+\nu_t d\Wd.
$$
Comparing the two equations it follows
$$
\nu_t d\Wd = \dot{c}\psip+\frac12c\lt\psip-\frac74c\tz  \nu_t\psim- c \nu_t \gamma.
$$
Wedging both sides by $\psim$ and using the fact that $\gamma\W\psim=0$ since $\gamma\in \Lambda^{2,1} \oplus \Lambda^{1,2}$, we get
\begin{equation}\label{prescoupled}
\nu_t d\Wd\W\psim = \frac23\dot{c}~\omega^3+\frac13c\lt\omega^3.
\end{equation}
Since for each $t$ it holds $d\Wd\W\psim = -|\Wd|^2\frac{\omega^3}{6},$ where the norm is induced by $h(t)$, equation \eqref{prescoupled} becomes
$$
-\nu_t |\Wd|^2\frac{\omega^3}{6} =\frac23\dot{c}~\omega^3+\frac13c\lt\omega^3
$$
and the following result is proved.
\begin{Proposition}
Suppose that the generalized Hitchin flow preserves the coupled condition $d\omega(t) = c(t)\psip(t)$, then the function $c(t)$ must evolve in the following way
\begin{equation*}
\ddt c(t) = -\frac12c(t)\lt-\frac14\nu_t |\Wd(t)|_{h(t)}^2.
\end{equation*}
Moreover, for each $t$ it must hold
\begin{equation*}
d\Wd = -\frac14|\Wd|^2\psip-\frac74c\tz\psim- c\gamma.
\end{equation*}
\end{Proposition}

In order to preserve the closedness of $\psip(t)$, we need
$$
d\left(\ddt\psip\right)=0.
$$
Moreover, taking the exterior derivative of both sides of the flow equation \eqref{flowpsip} of $\psip$ we have
$$
d\left(\ddt\psip\right) =-\frac74\tz \nu_t d\psim-\nu_t d\gamma.
$$
Comparing the two equations it then follows
\begin{equation}\label{dgamma}
d\gamma = -\frac74\tz \nu_t d\psim.
\end{equation}
Observe now that $d\gamma\W\omega =0$, since $\gamma$ is a primitive form of type $(2,1)+(1,2)$. Therefore, wedging both sides of \eqref{dgamma} by $\omega$ and recalling that 
$d\psim\W\omega=-\frac23c\omega^3$, we get
$$
\tz \nu_t c = 0,
$$
and then $\tz=0,$ since both $c$ and $\nu_t$ cannot be zero. In particular
\begin{equation*}
d\gamma=0.
\end{equation*}
We can summarize the results in the following
\begin{Proposition}
If the closedness of $\psip$ is preserved by the generalized Hitchin flow, then the intrinsic torsion form $\tz$ vanishes and the 3-form $\gamma$ is closed.
\end{Proposition}

Let us now consider the expression of $d\psim$ in \eqref{hf} and differentiate it with respect to $t$ having in mind the results already obtained:
$$
d\left(\ddt\psim\right) = \left(-\frac23\dot{c}-\frac43c\lt\right)\omega^2  +\left(-\frac43c\nu_t-\lt\right)\Wd\W\omega - \ddt\Wd\W\omega-\nu_t\Wd\W\Wd.
$$  
Taking the exterior derivative of both sides of the flow equation \eqref{flowpsim} of $\psim(t)$ we get
$$
d\left(\ddt\psim\right) = -\lt c\omega^2 -\frac32\lt \Wd\W\omega+\nu_t d(J\gamma).
$$
Comparing the two equations we obtain that the flow of $\Wd$ must obey the following equation
\begin{equation*}
\ddt\Wd\W\omega= \frac16\nu_t |\Wd|^2\omega^2  +\left(-\frac43c\nu_t+\frac12\lt\right)\Wd\W\omega -\nu_t\Wd\W\Wd-\nu_t d(J\gamma).
\end{equation*}

We also know that the following necessary conditions deriving from the Bianchi identity $d_7\hat{H}=0$ must hold
\begin{eqnarray}
dS^X&=&0,\label{bianchi1}\\
dS_t &=& \ddt S^X,\label{bianchi2}
\end{eqnarray}
where $\hat{H} = dt\W S_t + S^X$ is the component of the ten-dimensional flux along $M$.

\begin{remark}
The other constraint obtained from the Bianchi identities was recalled earlier, it is $d_7\tz=0$.
\end{remark}

Using the previous results, it follows from \cite{DLS} that for a coupled structure
$$
S^X = \nu_t^{-1}\tut\psim+J\gamma,\qquad S_t = 0.
$$
From the first identity \eqref{bianchi1} we then get 
\begin{equation}\label{djgamma}
d(J\gamma)=-\nu_t^{-1}\tut d\psim.
\end{equation}
Observe that $d(J\gamma)\W\omega = 0$. Thus, if we wedge both sides of \eqref{djgamma} by $\omega$ we obtain
$$\nu_t^{-1}\tut c=0,$$
from which follows $\tut=0$ and, as a consequence, $d(J\gamma)=0$.
The second identity \eqref{bianchi2} then reads 
\begin{equation*}
\ddt(J\gamma)=0.
\end{equation*}
We can summarize here some of the results obtained:
\begin{enumerate}[i)]
\item the only non-vanishing intrinsic torsion form of the $G_2$-structure after imposing all conditions is $\ttr$. Moreover, $*\ttr^X=\gamma$ and $\ttrt=0$.
\item $d\nu_t=0$.
\item $d\gamma$ and $d(J\gamma)=0$, thus $\gamma$ is harmonic.
\item $\lt = \frac23\nu_t c(t).$
\end{enumerate}
In particular, the evolution equations of the differential forms defining the coupled structure become
\begin{eqnarray}
\ddt\omega(t) &=& \frac23\nu_t c(t)\omega(t)+\nu_t\Wd(t),\nonumber\\
\ddt\psip(t) &=& \nu_t c(t)\psip(t)-\nu_t\gamma,\nonumber\\
\ddt\psim(t) &=&\nu_t c(t)\psim(t)+\nu_t J\gamma.\nonumber
\end{eqnarray}
Moreover, the intrinsic torsion forms of the coupled structure must evolve as 
\begin{eqnarray}
\ddt c(t) &=& -\frac13\nu_t (c(t))^2-\frac14\nu_t |\Wd(t)|_{h(t)}^2,\nonumber\\
\ddt\Wd(t)\W\omega(t)&=& \frac16\nu_t |\Wd(t)|_{h(t)}^2(\omega(t))^2  -\nu_t c(t)\Wd(t)\W\omega(t) -\nu_t(\Wd(t))^2,\nonumber
\end{eqnarray}
and for each $t$ it must hold
\begin{equation}
d\Wd = -\frac14|\Wd|^2\psip- c\gamma.\nonumber
\end{equation}

\subsection{The Hitchin flow as a particular case of the generalized Hitchin flow}
If we suppose that $\nu_t=1$ and $\gamma=0$, then 
$$\f = dt\W\omega+\psip$$
is a parallel $G_2$-structure. In this case, the evolution equations of the differential forms $\omega(t),\psip(t),\psim(t)$ read
\begin{equation}\label{restrictedforms}\renewcommand\arraystretch{1.4}
\begin{array}{rcl}
\ddt\omega(t) &=& \frac23c(t)\omega(t)+\Wd(t),\\
\ddt\psip(t) &=& c(t)\psip(t),\\
\ddt\psim(t) &=& c(t)\psim(t),
\end{array}
\end{equation}
the evolution equations of the intrinsic torsion forms of the coupled structure must be
\begin{equation}\label{restrictedtorsion}\renewcommand\arraystretch{1.4}
\begin{array}{rcl}
\ddt c(t) &=& -\frac13(c(t))^2-\frac14|\Wd(t)|_{h(t)}^2,\\
\ddt\Wd(t)\W\omega(t)&=& \frac16 |\Wd(t)|_{h(t)}^2(\omega(t))^2  -c(t)\Wd(t)\W\omega(t) -(\Wd(t))^2,
\end{array}
\end{equation}
and for each $t$ the 2-form $\Wd$ has to satisfy the following property
\begin{equation}\label{restricteddw2}
d\Wd = -\frac14|\Wd|^2\psip,
\end{equation}
which is one of the conditions widely discussed in Section \ref{cpdsusy}.

A solution of these equations which is coupled for each $t$ is also a coupled solution for the Hitchin flow equations and vice-versa.
For example, the coupled solution for the Hitchin flow on the Iwasawa Lie algebra $\mathfrak{I}$ obtained in the proof of Proposition \ref{HFnosol} satisfies \eqref{restrictedforms} 
and the conditions \eqref{restrictedtorsion}, \eqref{restricteddw2}. 
In the general case, the presence of $\Wd(t)$ in the flow equations makes rather complicated any attempt to solve them. However, we can show that a solution of them starting from 
a coupled SU(3)-structure stays coupled as long as it exists.
\begin{Proposition}
Let $(\omega(t),\psip(t),c(t),\Wd(t))$ be a solution of the equations \eqref{restrictedforms}, \eqref{restrictedtorsion}, \eqref{restricteddw2}, 
with initial condition a coupled structure $(\omega(0),\psip(0))$ satisfying $d\omega(0) =c(0)\psip(0)$. 
Then $(\omega(t),\psip(t))$ is a coupled structure as long as it exists.
\end{Proposition}
\begin{proof}
Consider $d\omega(t)-c(t)\psip(t)$, differentiating with respect to $t$ and using the hypothesis we get (omitting the $t$-dependence for brevity)
\begin{eqnarray}
\ddt(d\omega-c\psip) &=& d\left(\ddt\omega\right)-\dot{c}\psip-c\ddt\psip\nonumber\\
			         &=&  \frac23cd\omega+d\Wd+ \frac13c^2\psip+\frac14 |\Wd|^2\psip-c^2\psip\nonumber\\
			         &=&\frac23c(d\omega-c\psip).\nonumber
\end{eqnarray}
Thus, if we denote by $\rho(t) = d\omega(t)-c(t)\psip(t)$, we have that $\ddt\rho(t) = \frac23c(t)\rho(t)$. Therefore,
$\rho(t) = q(t)\rho(0),$ where $q(t) = e^{\int_0^t\frac23c(s)ds}$. But $\rho(0)=d\omega(0)-c(0)\psip(0) = 0$ since $(\omega(0),\psip(0))$ is coupled. 
Then $0=\rho(t) = d\omega(t)-c(t)\psip(t)$ and, as a consequence, $d\psip(t)=0$.
\end{proof}

\section{Conclusions}
In this paper, we considered from the mathematical point of view the properties of $\SU(3)$-structures which are of interest in the case of $\mathcal{N} = 1$ compactification of type 
IIA string theory to four-dimensional anti-de Sitter space on 6-manifolds endowed with an $\SU(3)$-structure, namely coupled structures satisfying (all or in part) the 
constraints given in \eqref{dw2proppsip} and \eqref{normcond}. 

First of all, we derived some properties of such structures and some constraints implied by them. These need to be taken into account when one looks for explicit examples.

We then turned our attention to examples of 6-manifolds endowed with this kind of $\SU(3)$-structures. 
In the case of nilmanifolds, we already knew that up to isomorphisms there are two non-abelian nilpotent Lie algebras admitting a coupled 
structure from \cite{FR}. Here, we showed that for only one of these the condition \eqref{dw2proppsip} is satisfied while the condition \eqref{normcond} cannot be ever satisfied. 
However, since in the physical setting conditions \eqref{dw2proppsip} and \eqref{normcond} can be relaxed in the presence of sources, the nilmanifolds generated by
 $\mathfrak{I}$ and $\mathfrak{N}$ may be used to construct examples of the considered type of compactification. 
This was done for the Iwasawa manifold in \cite{CKKLTZ}, thus it would be interesting to see what happens for the nilmanifold  corresponding to $\mathfrak{N}$. 
We also recalled an example firstly described in \cite{Tom}, this is of particular interest not only because it answers a question arising from \cite{Raf}, but also because allows to 
construct examples of $G_2$-structures with torsion inducing remarkable metrics.

In the last section, we considered the behaviour of the coupled condition with respect to the Hitchin flow and one of its possible generalizations determined starting from 
four-dimensional domain wall solutions of heterotic string theory preserving $\mathcal{N}=\frac12$ supersymmetry. 
We observed that it is not always possible to find coupled solutions for the Hitchin flow by working on explicit examples and derived the conditions implied by requiring that the 
coupled condition is preserved by the generalized Hitchin flow. 
An interesting open question would be to see whether there exist any example of a 1-parameter family of $\SU(3)$-structures which solves the Hitchin flow equations 
and is coupled for at least one but not for all $t$.

\end{document}